\documentclass{amsart}

\usepackage{hyperref}
\usepackage{amsmath,amsthm,amssymb}

\newtheorem{thm}{Theorem}
\newtheorem{prop}[thm]{Proposition}
\newtheorem{cor}[thm]{Corollary}
\newtheorem{lem}[thm]{Lemma}

\theoremstyle{definition}
\newtheorem{dfn}[thm]{Definition}
\newtheorem{rem}[thm]{Remark}
\newtheorem{ex}[thm]{Example}
\newtheorem*{ack}{Acknowledgments}

\numberwithin{thm}{section}
\numberwithin{equation}{section}

\DeclareMathOperator{\Ant}{Ant}

\title[Noncritical maps on GCBA spaces]{Noncritical maps on geodesically complete spaces with curvature bounded above}
\author[T. Fujioka]{Tadashi Fujioka}
\address{Department of Mathematics, Osaka University, Toyonaka, Osaka 560-0043, Japan}
\email{\href{mailto:fujioka@cr.math.sci.osaka-u.ac.jp}{fujioka@cr.math.sci.osaka-u.ac.jp}, \href{mailto:tfujioka210@gmail.com}{tfujioka210@gmail.com}}
\date{\today}
\thanks{Supported by JSPS KAKENHI Grant Numbers 15H05739, 20H00114, 22J00100}

\begin{document}

\begin{abstract}
We define and study the regularity of distance maps on geodesically complete spaces with curvature bounded above.
We prove that such a regular map is locally a Hurewicz fibration.
This regularity can be regarded as a dual concept of Perelman's regularity in the geometry of Alexandrov spaces with curvature bounded below.
As a corollary we obtain a sphere theorem for geodesically complete CAT($1$) spaces.
\end{abstract}

\maketitle

\section{Introduction}\label{sec:int}

The notions of lower and upper curvature bounds for metric spaces were introduced by Alexandrov using triangle comparison.
More specifically, a metric space has curvature $\ge\kappa$ (resp.\ $\le\kappa$) if any small geodesic triangle is ``thicker'' (resp.\ ``thinner'') than the geodesic triangle with the same sidelengths on the plane of constant curvature $\kappa$.
Metric spaces with curvature bounded below and above are abbreviated as CBB spaces and CBA spaces, respectively.
The properties of both curvature bounds are completely different in general.

CBB spaces are usually called Alexandrov spaces and play an essential role in the convergence theory of Riemannian manifolds with sectional curvature bounded below.
The fundamental theory of finite-dimensional CBB spaces was developed by Burago, Gromov, and Perelman \cite{BGP}.
They showed that topological dimension and Hausdorff dimension coincide for CBB spaces and that any CBB space contains an open dense Lipschitz manifold of its dimension.
To prove these, they defined a strainer (a collection of points around a point satisfying an orthogonal condition) and studied the distance map from points of a strainer, namely a strainer map.
Perelman \cite{Per:alex}, \cite{Per:mor} then developed the structure theory of CBB spaces by introducing the regularity of distance maps more general than that of strainer maps, which extends the regularity of distance functions introduced by Grove and Shiohama \cite{GS} in the Riemannian setting.
In particular, he proved that such a regular map is locally a bundle map and obtained a stratification of an arbitrary CBB space into topological manifolds.

On the other hand, the structure of CBA spaces is much more complicated (\cite{Kl}).
For example, the Hausdorff dimension may be bigger than the topological dimension.
To obtain some control, we assume the \textit{(local) geodesic completeness}, that is, the extension property of geodesics.
A separable, locally compact, locally geodesically complete CBA space is called a \textit{GCBA space}.
Lytchak and Nagano \cite{LN:geod}, \cite{LN:top} recently published the fundamental theory of GCBA spaces (note also that there was an unpublished work of Otsu and Tanoue \cite{OT}).
Their results suggest that the geometry of GCBA spaces has many parallels with that of CBB spaces.
They showed the coincidence of topological and Hausdorff dimensions for GCBA spaces and proved that any GCBA space can be stratified in a measure-theoretic sense so that each stratum contains an open dense Lipschitz manifold of its dimension.
They also obtained a manifold recognition theorem for GCBA spaces.
Their main technical tool is a strainer on a GCBA space, which can be seen as a dual of a strainer on a CBB space.
In particular, they showed that any strainer map is locally a Hurewicz fibration (but not a bundle map; see \cite[2.7]{N:asym} for example).

In this paper we define and study the regularity of distance maps on GCBA spaces more general than that of strainer maps, which can be regarded as a dual concept of Perelman's regularity in CBB geometry.
To define it, we introduce the following notation.
Let $\Sigma$ be a compact, geodesically complete CAT($1$) space with diameter $\pi$ (any space of directions of a GCBA space satisfies these conditions).
We define the \textit{antipodal distance} $\overline{|\ ,\ |}$ on $\Sigma$ by
\[\overline{|\xi\eta|}:=\sup_{x\in\Sigma}|\xi x|+|\eta x|-\pi\]
for $\xi,\eta\in \Sigma$, where $|\ ,\ |$ denotes the distance on $\Sigma$.
The antipodal distance is not a distance in the usual sense.
The reason for this name is the following equivalent definition.
Let $\Ant(\xi)$ denote the set of all antipodes $\bar\xi$ of $\xi$, i.e.\ $|\xi\bar\xi|=\pi$.
Then it easily follows from the geodesic completeness that
\[\overline{|\xi\eta|}=\sup_{\bar\xi\in\Ant(\xi)}|\bar\xi\eta|=\sup_{\bar\eta\in\Ant(\eta)}|\xi\bar\eta|\]
(see Lemma \ref{lem:ant} for the detail).
Thus the antipodal distance on spaces of directions reflects the branching phenomena of shortest paths in GCBA spaces.

Let $U$ be a tiny ball in a GCBA space $X$, that is, a small metric ball where triangle comparison holds (see Sec.\ \ref{sec:gcba} for the precise definition).
In this paper we usually work inside some tiny ball.
For $p,a\in U$, we denote by $\Sigma_p$ the space of directions at $p$ and by $a'_p\in\Sigma_p$ (or simply $a'$ if no confusion arises) the direction of the unique shortest path from $p$ to $a$.
Using the above notation, the definition of a strainer by Lytchak-Nagano \cite[7.2]{LN:geod} can be expressed as follows: $\{a_i\}_{i=1}^k$ in $U$ is a $(k,\delta)$-strainer at $p\in U$ if there exists $\{b_i\}_{i=1}^k$ in $U$ such that
\[\overline{|a_i'b_i'|}<\delta,\quad|a_i'a_j'|,|a_i'b_j'|,|b_i'b_j'|<\pi/2+\delta\]
in $\Sigma_p$ for any $i\neq j$.
Here the first inequality guarantees that there is an almost unique extension of the shortest path $a_ip $.
On the other hand, a strainer in a CBB space is defined by the inequalities $|a_i'b_i'|>\pi-\delta$ and $|a_i'a_j'|,|a_i'b_j'|,|b_i'b_j'|>\pi/2-\delta$ (for any choice of directions), where the first inequality means that $a_ip$ is almost extendable as a shortest path.
Note that shortest paths in CBB spaces are neither unique nor extendable in general, but do not branch, whereas shortest paths in GCBA spaces are unique and extendable at least in tiny balls, but may branch.

We now define the regularity of distance maps dealt with in this paper.
Let $\varepsilon$ and $\delta$ be small positive numbers such that $\delta\ll\varepsilon$, where the choice of $\delta$ depends only on the (local) dimension and $\varepsilon$ (more precisely, it will be determined by the proof of each statement; see Sec. \ref{sec:nac}).

\begin{dfn}\label{dfn:nc}
Let $U$ be a tiny ball in a GCBA space $X$ and $a_i\in U$ ($1\le i \le k$).
We say that $f=(|a_1\cdot|,\dots,|a_k\cdot|):U\to\mathbb R^k$ is \textit{$(\varepsilon,\delta)$-noncritical} at $p\in U$ (distinct from $a_i$) if
\begin{enumerate}
\item $\overline{|a_i'a_j'|}<\pi/2+\delta$ in $\Sigma_p$ for any $1\le i\neq j\le k$;
\item there exists $b\in U$ (distinct from $p$) such that $\overline{|a_i'b'|}<\pi/2-\varepsilon$ in $\Sigma_p$ for any $1\le i\le k$.
\end{enumerate}
We simply say $f$ is \textit{noncritical} at $p$ if it is $(\varepsilon,\delta)$-noncritical at $p$ for some $\delta\ll\varepsilon$.
\end{dfn}

\begin{rem}
The original definition \cite[3.1]{Per:alex} of $(\varepsilon,\delta)$-noncriticality in the CBB setting requires $|a_i'a_j'|>\pi/2-\delta$ and $|a_i'b'|>\pi/2+\varepsilon$ (for any choice of directions).
In case $\Sigma_p$ is a unit sphere and $a_i'$, $b'$ are unique, both definitions coincide.
Note that in this case $a_i'$ are linearly independent as vectors in Euclidean space.
\end{rem}

\begin{rem}
One can also define the \textit{$\varepsilon$-regularity} by strengthening the condition (1) to $\overline{|a_i'a_j'|}<\pi/2-\varepsilon$, as Perelman did in \cite{Per:mor}, and simplify some of the proofs.
The error $\delta$ is only used in Lemma \ref{lem:comp}.
In case $k=1$ there is no difference.
\end{rem}

Note that being a noncritical point of $f$ is an open condition by the upper semicontinuity of angle and the local geodesic completeness.
It is easy to see that if $\{a_i\}_{i=1}^k $ is a $(k,\delta)$-strainer at $p$, then $f=(|a_1\cdot|,\dots,|a_k\cdot|)$ is $(c,2\delta)$-noncritical at $p$ for some constant $c>0$ (use the $c$-openness of $f$ to find $b'$; see \cite[8.2]{LN:geod}).

\begin{ex}\label{ex:nc}
Let $X$ be the Euclidean cone over the circle of length $2\pi+\theta$, which is a geodesically complete CAT($0$) space for $\theta\ge 0$.
Then,
\begin{enumerate}
\item if $\theta<\pi/4$, there exists a noncritical map $f:X\to\mathbb R^2$ at the vertex $o$;
\item if $\pi/4\le\theta<\pi$, there exists a noncritical function $f:X\to\mathbb R$ at $o$, but no noncritical map $f:X\to\mathbb R^2$ at $o$;
\item if $\theta\ge\pi$, there exists no noncritical function $f:X\to\mathbb R$ at $o$.
\end{enumerate}
\end{ex}

\begin{ex}\label{ex:iso}
Let $p$ be an isolated singularity of a GCBA space, that is, a non-manifold point such that its punctured neighborhood is a manifold.
Then $\overline{|\xi\eta|}=\pi$ for any $\xi,\eta\in\Sigma_p$.
In particular, there exists no noncritical function at $p$.
Indeed, if $\overline{|\xi\eta|}<\pi$, then $\Sigma_p$ is covered by the two open balls of radius $\pi$ centered at $\xi$ and $\eta$, both of which are contractible.
Since $\Sigma_p$ is a homology manifold (\cite[3.3, 3.4]{LN:top}), the same argument as in \cite[8.2]{LN:top} shows that $\Sigma_p$ has the homotopy type of a sphere.
Therefore the theorem of Lytchak-Nagano \cite[1.1]{LN:top} implies that $p$ is a manifold point, which is a contradiction.
\end{ex}

We prove the following two theorems generalizing the results of Lytchak-Nagano for strainer maps.
The first extends \cite[8.2, 11.2]{LN:geod} and the second extends \cite[5.1]{LN:top}.
As before, $U$ denotes a tiny ball in a GCBA space $X$.
Let $T_p$ denote the tangent cone at $p\in U$.
Note that the local dimension $\dim T_p$ may not be constant.
We denote by $c(\varepsilon)$ a positive constant depending only on the (local) dimension and $\varepsilon$.

\begin{thm}\label{thm:open}
Let $f:U\to\mathbb R^k$ be an $(\varepsilon,\delta)$-noncritical map at $p\in U$ for $\delta\ll\varepsilon$.
Then $k\le\dim T_p$ and $f$ is $c(\varepsilon)$-open near $p$.
Furthermore if $k=\dim T_p$, then $f$ is a bi-Lipschitz open embedding near $p$.
\end{thm}

\begin{thm}\label{thm:hure}
Let $f:U\to\mathbb R^k$ be a noncritical map at $p\in U$, where $k<\dim T_p$.
Then there exists an arbitrarily small contractible open neighborhood $V$ of $p$ such that $f:V\to f(V)$ is a Hurewicz fibration with contractible fibers.
\end{thm}

\begin{rem}
The above theorems hold for noncritical maps on CBB spaces (note that the local dimension is constant for CBB spaces).
Moreover, any $(\varepsilon,\delta)$-noncritical map at $p$ in a CBB space is a bundle map near $p$ with conical fibers, provided that $\delta$ is sufficiently small compared to the volume of $\Sigma_p$.
See \cite{Per:alex} for the details (cf.\ \cite{Per:mor}).
\end{rem}

As a corollary we obtain the following sphere theorem.
For other sphere theorems, see \cite{N:vol} and the references therein.
We may assume that the diameter of a geodesically complete CAT($1$) space is exactly $\pi$ by considering the $\pi$-truncated metric if necessary (see also Remark \ref{rem:bilip}).

\begin{cor}\label{cor:sph}
Let $\Sigma$ be a compact, geodesically complete CAT(1) space of dimension $n$.
Assume that there exist $\{\xi_i\}_{i=1}^{n+1}$ and $\eta$ in $\Sigma$ such that
\[\overline{|\xi_i\xi_j|}<\pi/2+\delta,\quad\overline{|\xi_i\eta|}<\pi/2-\epsilon\]
for any $i\neq j$ and $\delta\ll\varepsilon$.
Then $\Sigma$ is bi-Lipschitz homeomorphic to $\mathbb S^n$.
\end{cor}

\begin{rem}
The above estimate is optimal in the following sense.
Let $T$ be a tripod, that is, a metric space consisting of three points with pairwise distance $\pi$.
Let $\Sigma$ be the spherical join between $\mathbb S^{n-1}$ and $T$, which is not homeomorphic to $\mathbb S^n$.
We regard $\mathbb S^{n-1}$ and $T$ as isometrically embedded in $\Sigma$.
Choose $\{\xi_i\}_{i=1}^{n+1}\subset\mathbb S^{n-1}$ such that $\overline{|\xi_i\xi_j|}<\pi/2$ and $\eta\in T$.
Then we have $\overline{|\xi_i\eta|}=\pi/2$ for all $i$.
\end{rem}

\begin{rem}
Lytchak-Nagano \cite[1.5]{LN:top} also proved the following sphere theorem:
if $\Sigma$ is a compact, geodesically complete CAT(1) space with no tripods, then it is homeomorphic to a sphere.
The author does not know whether the assumption of the above corollary implies the absence of tripods.
\end{rem}

\begin{rem}
There is the CBB counterpart of the above corollary:
if $\Sigma$ is a CBB($1$) space of dimension $n$ and if $\{\xi_i\}_{i=1}^{n+1}$ and $\eta$ in $\Sigma$ satisfy $|\xi_i\xi_j|>\pi/2-\delta$ and $|\xi_i\eta|>\pi/2+\epsilon$, then $\Sigma$ is bi-Lipschitz homeomorphic to $\mathbb S^n$.
Moreover, if there exist $\{\xi_i\}_{i=1}^k$ and $\eta$ in $\Sigma$ satisfying the same inequalities, where $k\le n$, then $\Sigma$ is homeomorphic to a $k$-fold suspension (\cite[4.5]{Per:alex}, \cite[Theorem C]{GW}).
\end{rem}

This paper is organized as follows.
In Sec.\ \ref{sec:nac} we introduce some notation used in this paper.
In Sec.\ \ref{sec:pre} we give preliminaries on GCBA spaces, $\varepsilon$-open maps, and Hurewicz fibrations.
In Sec.\ \ref{sec:inf} we study the properties of the differential of a noncritical map.
In Sec.\ \ref{sec:loc} we first prove Theorem \ref{thm:open} by using the results of the previous section.
We then construct a local neighborhood retraction to the fiber of a noncritical map to prove Theorem \ref{thm:hure}.
Finally we give the proof of Corollary \ref{cor:sph}.

\section{Notation and conventions}\label{sec:nac}

We will use the following standard notation.
For points $p$, $q$ in a metric space, $|pq|$ denotes the distance between them.
For $r>0$, we denote by $B(p,r)$ (resp.\ $\bar B(p,r)$) the open (resp.\ closed) metric ball of radius $r$ centered at $p$.
The boundary $\partial B(p,r)$ is defined by the difference $\bar B(p,r)\setminus B(p,r)$.

We will also use the following notation from \cite{Per:alex}.
As in the introduction, $\varepsilon$ and $\delta$ denote positive numbers such that $\delta\ll\varepsilon$.
The choice of $\delta$ depends only on the (local) dimension, the upper curvature bound, and $\varepsilon$ (the dependence on the upper curvature bound is not necessary if it is taken to be nonnegative).
Whenever $\varepsilon$ and $\delta$ appear in a statement, it means that the statement holds for a suitable choice of $\delta$ depending on $\varepsilon$, which will be determined by the proof.
We denote by $c(\varepsilon)$ various positive constants such that $c(\varepsilon)\ll\varepsilon$, and by $\varkappa(\delta)$ various positive functions such that $\varkappa(\delta)\to0$ as $\delta\to 0$.
They also depend only on the (local) dimension, the upper curvature bound, and $\varepsilon$.
In particular we may assume $\varkappa(\delta)\ll c(\varepsilon)$ by taking $\delta\ll\varepsilon$.
Whenever $c(\varepsilon)$ and $\varkappa(\delta)$ appear in a statement, it means that the statement holds for some $c(\varepsilon)$ and $\varkappa(\delta)$ determined by the proof.

\section{Preliminaries}\label{sec:pre}

\subsection{GCBA spaces}\label{sec:gcba}

Here we recall basic notions and facts about GCBA spaces.
We refer the reader to \cite{LN:geod} and \cite{BBI} for more details.
We assume all metric spaces are separable and locally compact unless otherwise stated.

Let $\kappa\in\mathbb R$.
We denote by $S^2_\kappa$ the complete simply-connected surface of constant curvature $\kappa$ and by $D_\kappa$ the diameter of $S^2_\kappa$.
A complete metric space is called a \textit{CAT($\kappa$) space} if any two points with distance $<D_\kappa$ can be joined by a shortest path and if any geodesic triangle with perimeter $<2D_\kappa$ is not thicker than the comparison triangle on $S^2_\kappa$.
A metric space is called a \textit{CBA($\kappa$) space} if any point has a CAT($\kappa$) neighborhood.
For example, a complete Riemannian manifold is CBA($\kappa$) if and only if its sectional curvature is bounded above by $\kappa$ and is CAT($\kappa$) if in addition its injectivity radius is bounded below by $D_\kappa$.
One can also construct numerous examples by Reshetnyak's gluing theorem: a gluing of two CBA($\kappa$) (resp.\ CAT($\kappa$)) spaces along their isometric convex subsets is again a CBA($\kappa$) (resp.\ CAT($\kappa$)) space.

Let $X$ be a CBA($\kappa$) space.
A geodesic is a curve that is locally a shortest path.
We say that $X$ is \textit{locally geodesically complete} if any geodesic can be extended to a geodesic beyond its endpoints, and that $X$ is \textit{geodesically complete} if the extension can be defined on $\mathbb R$.
For example, if a small punctured ball at each point of $X$ is noncontractible, then $X$ is locally geodesically complete.
In particular, any homology manifold with a CBA metric is locally geodesically complete.
If $X$ is complete, then the local geodesic completeness is equivalent to the geodesic completeness.
A separable, locally compact, locally geodesically complete CBA space is called a \textit{GCBA space}.

Let $X$ be a GCBA($\kappa$) space.
The angle between two shortest paths is defined by the limit of comparison angles.
The \textit{space of directions} at $p$, denoted by $\Sigma_p$, is the set of the directions of shortest paths emanating from $p$ equipped with the angle metric.
$\Sigma_p$ is a compact, geodesically complete CAT($1$) space with diameter $\pi$.
By the local geodesic completeness, for any given direction, there exists a shortest path starting in that direction.
Furthermore any direction has at least one opposite direction.
The \textit{tangent cone} $T_p$ at $p$ is the Euclidean cone over $\Sigma_p$.
$T_p$ is isometric to the blow-up limit of $X$ at $p$ and is a geodesically complete CAT($0$) space.

The \textit{dimension} of a GCBA space is defined by the Hausdorff dimension, which coincides with the topological dimension.
Note that the local dimension $\dim T_p$ is finite and upper semicontinuous, but not necessarily constant.

We say that a metric ball $U$ in $X$ of radius $r$ is \textit{tiny} if the closed concentric ball of radius $10r$ is a compact CAT($\kappa$) space and $r<\min\{D_\kappa/100,1\}$.
In this paper we usually work inside some tiny ball.
Any two points in $U$ are joined by a unique shortest path contained in $U$ and any shortest path in $U$ can be extended to a shortest path of length $9r$.
The angle is upper semicontinuous in $U$.
For $p,a\in U$, we denote by $a'_p\in\Sigma_p$ (or simply $a'$ if no confusion arises) the direction of the unique shortest path from $p$ to $a$.
Let $f=|a\cdot|$.
Then $f$ is convex on $U$ and the directional derivative $f'$ of $f$ on $\Sigma_p$ is given by the first variation formula $f'=-\cos|a'\cdot|$.

The directional derivative is defined as follows.
Let $f$ be a (locally Lipschitz) function defined on an open subset $U$ of a GCBA space.
For $p\in U$ and $\xi\in\Sigma_p$, let $\gamma_\xi(t)$ denote a shortest path starting at $p$ in the direction $\xi$ and parametrized by arclength.
The \textit{directional derivative} $f'(\xi)$ of $f$ in the direction $\xi$ is defined by $\lim_{t\to 0}t^{-1}(f(\gamma_\xi(t))-f(p))$, if the limit exists and is independent of the choice of a shortest path $\gamma_\xi$.

\subsection{\boldmath{$\varepsilon$}-open maps}

Let $f:X\to Y$ be a continuous map between metric spaces and let $\varepsilon$ be a (small) positive number.
We say that $f$ is \textit{$\varepsilon$-open} if for any $x\in X$ and any sufficiently small $r>0$, we have $B(f(x),\varepsilon r)\subset f(B(x,r))$.

We will use the following two lemmas from \cite{Per:alex} regarding $\varepsilon$-open maps.
The proofs are straightforward and do not rely on any curvature assumption.
A map from an open subset of a GCBA space to Euclidean space is called \textit{differentiable} if each coordinate function has directional derivative in any direction.

\begin{lem}[{\cite[2.1.2]{Per:alex}}]\label{lem:open}
Let $U$ be an open subset of a GCBA space $X$.
Suppose that a differentiable map $f=(f_1,\dots,f_k):U\to\mathbb R^k$ satisfies the following property:
for any $p\in U$,
\begin{enumerate}
\item there exists $\xi_i\in\Sigma_p$ ($1\le i\le k$) such that $f_i'(\xi_i)<-\varepsilon$ and $|f_j'(\xi_i)|<\delta$ for any $j\neq i$, where $\delta\ll\varepsilon$;
\item there exists $\eta\in\Sigma_p$ such that $\varepsilon<f_j'(\eta)<\varepsilon^{-1}$ for any $j$.
\end{enumerate}
Then $f$ is $c(\varepsilon)$-open on $U$ with respect to the Euclidean norm in $\mathbb R^k$.
The choices of $\delta$ and $c(\varepsilon)$ depend only on $k$ and $\varepsilon$.
\end{lem}

\begin{proof}
We may use the $1$-norm in $\mathbb R^k$.
It suffices to show that for any $p\in U$ and any $v\in\mathbb R^k$ sufficiently close to $f(p)$, there exists $q\in U$ arbitrarily close to $p$ such that $|f(q)v|\le|f(p)v|-c(\varepsilon)|pq|$.
Then one can find $r\in U$ such that $f(r)=v$ and $c(\varepsilon)|pr|\le|f(p)v|$ by a standard argument, which completes the proof.
Here and below $c(\varepsilon)$ denotes various positive constants such that $c(\varepsilon)\ll\varepsilon$ and $\varkappa(\delta)$ denotes various positive function such that $\varkappa(\delta)\to0$ as $\delta\to0$ (see Sec.\ \ref{sec:nac}).

We may assume $v=0$.
If $f_i(p)>0$ for some $i$, then the claim follows from the assumption (1).
Similarly, if $f_i(p)<0$ for all $i$, then the claim follows from the assumption (2).
Hence we may assume $f_i(p)\le 0$ for all $i$ and $f_i(p)=0$ for some $i$.
For simplicity, we assume $f_1(p)=f_2(p)=0$ and $f_i(p)<0$ for $i\ge 3$.

First, by the assumption (2), we choose $q_1\in U$ near $p$ such that $\varepsilon|pq_1|<f_i(q_1)<\varepsilon^{-1}|pq_1|$ for $i=1,2$ and $f_i(p)+\varepsilon|pq_1|<f_i(q_1)<-|pq_1|$ for $i\ge 3$.
We consider a closed subset
\begin{equation*}
A_1:=\left\{x\in U\;\middle|\;
\begin{gathered}
0\le f_1(x)\le f_1(q_1)-c(\varepsilon)|xq_1|\\
|f_i(x)-f_i(q_1)|\le\varkappa(\delta)(f_1(q_1)-f_1(x))\ (i\ge2)
\end{gathered}
\right\}.
\end{equation*}

By a standard argument using the assumption (1), one can find $q_2\in A_1$ such that $f_1(q_2)=0$.
Then we have $|q_2q_1|\le c(\varepsilon)^{-1}f_1(q_1)\le c(\varepsilon)^{-1}|pq_1|$.
Similarly we have $|f_i(q_2)-f_i(q_1)|\le\varkappa(\delta)f_1(q_1)\le\varkappa(\delta)|pq_1|$ for $i\ge 2$.
In particular $0<f_2(q_2)<c(\varepsilon)^{-1}|pq_1|$.

Next we consider
\begin{equation*}
A_2:=\left\{x\in U\;\middle|\;
\begin{gathered}
0\le f_2(x)\le f_2(q_2)-c(\varepsilon)|xq_2|\\
|f_i(x)-f_i(q_1)|\le\varkappa(\delta)(f_2(q_2)-f_2(x))\ (i\neq2)
\end{gathered}
\right\}.
\end{equation*}
Using the assumption (1), one can find $q\in A_2$ such that $f_2(q)=0$.
As above we obtain $|qq_2|\le c(\varepsilon)^{-1}|pq_1|$ and $|f_i(q)-f_i(q_2)|\le\varkappa(\delta)|pq_1|$ for $i\neq 2$.
In particular $|f_1(q)|<\varkappa(\delta)|pq_1|$.

For $i\ge3$, we have
\begin{align*}
f_i(q)&\le f_i(q_1)+\varkappa(\delta)|pq_1|\le-|pq_1|+\varkappa(\delta)|pq_1|<0,\\
f_i(q)&\ge f_i(q_1)-\varkappa(\delta)|pq_1|\\
&\ge f_i(p)+\varepsilon|pq_1|-\varkappa(\delta)|pq_1|\\
&\ge f_i(p)+c(\varepsilon)|pq_1|.
\end{align*}
We also have $|pq|\le c(\varepsilon)^{-1}|pq_1|$.
Hence we obtain $|f(q)|\le|f(p)|-c(\varepsilon)|pq_1|\le|f(p)|-c(\varepsilon)|pq|$, which completes the proof.
\end{proof}

\begin{lem}[{\cite[2.1.3]{Per:alex}}]\label{lem:dir}
Let $f:U\to\mathbb R^k$ be a differentiable $\varepsilon$-open map defined on an open subset $U$ of a GCBA space $X$.
Let $p\in U$, $\xi\in\Sigma_p$ be such that $f'(\xi)=0$.
Then there exists $q\in f^{-1}(f(p))$ arbitrarily close to $p$ such that $q'$ is arbitrarily close to $\xi$.
In particular, if $\{g_i\}_i$ is a finite collection of differentiable, locally Lipschitz functions on $U$, then one can choose $q$ so that $g_i(q)>g_i(p)$ if $g_i'(\xi)>0$ and $g_i(q)<g_i(p)$ if $g_i'(\xi)<0$.
\end{lem}

\begin{proof}
Choose a point $q_1\in U$ near $p$ on a shortest path starting in the direction $\xi$ so that $|f(p)f(q_1)|<\delta|pq_1|$, where $\delta\ll\varepsilon$.
Using the $\varepsilon$-openness of $f$, we can find $q\in U$ such that $f(q)=f(p)$ and $\varepsilon|qq_1|\le|f(p)f(q_1)|$.
In particular $|qq_1|<\varkappa(\delta)|pq_1|$ and hence $\angle qpq_1<\varkappa(\delta)$, which completes the proof of the first half.

Let us show the second half.
If $g_i'(\xi)>0$, we may assume $g_i(q_1)>g_i(p)+l|pq_1|$ for some $l>0$.
Since $g_i$ is locally Lipshitz, we have $|g_i(q)-g_i(q_1)|\le L|qq_1|$ for some $L>0$.
Together with $|qq_1|\ll|pq_1|$, these imply $g_i(q)>g_i(p)$.
\end{proof}

\subsection{Hurewicz fibrations}

We assume all maps are continuous.
A map between topological spaces is called a \textit{Hurewicz fibration} if it satisfies the homotopy lifting property with respect to any space.

The following two theorems from geometric topology, used by Lytchak-Nagano \cite{LN:top} for strainer maps, provide sufficient conditions for a map to be a Hurewicz fibration.
Both are due to Ungar \cite{U} and based on Michael's selection theorem.

\begin{dfn}\label{dfn:lucf}
Let $f:X\to Y$ be a map between metric spaces.
We say that $f$ has \textit{locally uniformly contractible fibers} if the following holds:
for any $x\in X$ and every neighborhood $U$ of $x$, there exists a neighborhood $V\subset U$ of $x$ such that for any fiber $\Pi$ of $f$ intersecting $V$, the intersection $\Pi\cap V$ is contractible in $\Pi\cap U$.
\end{dfn}

\begin{thm}[{\cite[Theorem 1]{U}}]\label{thm:glo}
Let $X$, $Y$ be finite-dimensional, compact metric spaces and let $Y$ be an ANR.
Let $f:X\to Y$ be an open, surjective map with locally uniformly contractible fibers.
Then $f$ is a Hurewicz fibration.
\end{thm}

\begin{thm}[{\cite[Theorem 2]{U}}]\label{thm:loc}
Let $X$, $Y$ be finite-dimensional, locally compact metric spaces.
Let $f:X\to Y$ be an open, surjective map with locally uniformly contractible fibers.
Assume all fibers of $f$ are contractible.
Then $f$ is a Hurewicz fibration.
\end{thm}

\section{Infinitesimal properties}\label{sec:inf}

In this section we study the infinitesimal properties of noncritical maps and prove some lemmas which will be used in the next section.
This section corresponds to \cite[\S2]{Per:alex} (cf.\ \cite[\S2]{Per:mor}) in CBB geometry.
We use the notation $\varepsilon$, $\delta$, $c(\varepsilon)$, and $\varkappa(\delta)$ introduced in Sec.\ \ref{sec:nac}.
The choice of $\delta$ in this section will be determined by the proof of Proposition \ref{prop:dir}.

Throughout this section, $\Sigma$ denotes a compact, geodesically complete CAT($1$) space with diameter $\pi$.
In case $\dim\Sigma=0$, we assume $\Sigma$ is a set of finite points with pairwise distance $\pi$ and not a singleton (for the sake of induction).
Note that any two points in $\Sigma$ with distance $<\pi$ can be joined by a unique shortest path and that any shortest path in $\Sigma$ can be extended to a shortest path of length $\pi$.
As in the introduction, we define
\begin{gather*}
\overline{|\xi\eta|}:=\sup_{x\in\Sigma}|\xi x|+|\eta x|-\pi\\
\Ant(\xi):=\{\bar\xi\in\Sigma\mid|\xi\bar\xi|=\pi\}
\end{gather*}
for $\xi,\eta\in\Sigma$.
We call $\bar\xi\in\Ant(\xi)$ an \textit{antipode} of $\xi$ (not necessarily unique).

We first show the following fact mentioned in the introduction.

\begin{lem}\label{lem:ant}
\[\overline{|\xi\eta|}=\sup_{\bar\xi\in\Ant(\xi)}|\bar\xi\eta|\]
\end{lem}

\begin{proof}
The inequality $\overline{|\xi\eta|}\ge\sup_{\bar\xi\in\Ant(\xi)}|\bar\xi\eta|$ is clear from the definition.
For any $x\in\Sigma$, extend the shortest path $\xi x$ beyond $x$ to a shortest path $\xi\bar\xi$ of length $\pi$.
Then we have
\[|\xi x|+|\eta x|-\pi=|\eta x|-|\bar\xi x|\le|\eta\bar\xi|,\]
which implies the opposite inequality (we may assume $x\neq\xi$ and $x\notin\Ant(\xi)$).
\end{proof}

This section deals with the infinitesimal version of an $(\varepsilon,\delta)$-noncritical map.

\begin{dfn}\label{dfn:nci}
A collection $\{\xi_i\}_{i=1}^k$ of points of $\Sigma$ is said to be \textit{$(\varepsilon,\delta)$-noncritical} (or simply \textit{noncritical}) if the following hold:
\begin{enumerate}
\item $\overline{|\xi_i\xi_j|}<\pi/2+\delta$ for any $i\neq j$;
\item there exists $\eta\in\Sigma$ such that $\overline{|\xi_i\eta|}<\pi/2-\varepsilon$ for any $i$.
\end{enumerate}
In this case we call $\eta$ a \textit{regular direction} for $\{\xi_i\}_{i=1}^k$.
\end{dfn}

\begin{rem}\label{rem:ineq}
Lemma \ref{lem:ant} and the triangle inequality immediately imply
\[|\xi_i\xi_j|>\pi/2-\delta,\quad|\xi_i\eta|>\pi/2+\varepsilon\]
for any $i\neq j$.
Furthermore if $k\ge2$, these inequalities and Definition \ref{dfn:nci}(2) imply $|\xi_i\xi_j|<\pi-2\varepsilon$ and $|\xi_i\eta|<\pi-\varepsilon/2$ (provided $\delta<\varepsilon/2$).
\end{rem}

The next lemma enables us to use induction on the dimension of $\Sigma$.

\begin{lem}\label{lem:ind}
Suppose $\dim\Sigma\ge1$.
Let $\{\xi_i\}_{i=1}^k$ be an $(\varepsilon,\delta)$-noncritical collection in $\Sigma$ with a regular direction $\eta$.
For $x\in\Sigma$, assume either
\begin{enumerate}
\item $\xi_i,\eta\in B(x,\pi/2+\delta)$ for any $i$; or
\item $\xi_i,\eta\notin B(x,\pi/2-\delta)$ for any $i$.
\end{enumerate}
Then, $\{\xi_i'\}_{i=1}^k$ in $\Sigma_x$ is a $(c(\varepsilon),\varkappa(\delta))$-noncritical collection with a regular direction $\eta'$ (see Sec.\ref{sec:nac} for the notation).
\end{lem}

\begin{proof}
We assume (1).
Note that this assumption together with $|\xi_i\eta|>\pi/2+\varepsilon$ implies $\xi_i,\eta\notin B(x,\varepsilon/2)$ (provided $\delta<\varepsilon/2$).
Let $\bar\xi_i'$ be an arbitrary antipode of $\xi_i'$ in $\Sigma_x$.
Let $\bar\xi_i$ be the point on the shortest path starting in the direction $\bar\xi_i'$ at distance $\pi-|x\xi_i|$ from $x$.
Then triangle comparison shows that $\bar\xi_i$ is an antipode of $\xi_i$.
Thus we have $|\bar\xi_i\xi_j|<\pi/2+\delta$ by the assumption of noncriticality.
By triangle comparison again, we have
\[\cos|\bar\xi_i'\xi_j'|\ge\frac{\cos|\bar\xi_i\xi_j|-\cos|x\bar\xi_i|\cos|x\xi_j|}{\sin|x\bar\xi_i|\sin|x\xi_j|}>-\varkappa(\delta),\]
where $\pi/2-\delta<|x\bar\xi_i|\le\pi-\varepsilon/2$ and $\varepsilon/2\le|x\xi_j|<\pi/2+\delta$ by the assumption (1).
Similarly since $|\bar\xi_i\eta|<\pi/2-\varepsilon$, we have $\cos|\bar\xi_i'\eta'|\ge c(\varepsilon)$.
The case (2) is similar (this assumption together with $\overline{|\xi_i\eta|}<\pi/2-\varepsilon$ implies $\xi_i,\eta\in B(x,\pi-\varepsilon/2)$).
\end{proof}

The main result of this section is the following counterpart of \cite[2.2, 2.3]{Per:alex} (cf.\ \cite[2.3]{Per:mor}) in CBB geometry.

\begin{prop}\label{prop:dir}
Let $\{\xi_i\}_{i=1}^k$ be an $(\varepsilon,\delta)$-noncritical collection in $\Sigma$ with a regular direction $\eta$, where $\delta\ll\varepsilon$.
Then
\begin{enumerate}
\item $k\le\dim\Sigma+1$;
\item there exists $v\in\Sigma$ such that
\[|v\xi_1|<\pi/2-c(\varepsilon),\quad|v\xi_i|=\pi/2,\quad|v\eta|>\pi/2+c(\varepsilon)\]
for any $i\ge2$.
\end{enumerate}
\end{prop}

\begin{proof}
We first prove (1) by induction on $\dim\Sigma$.
If $\dim\Sigma=0$, then clearly $k\le1$ and $k=1$ is possible only when $\Sigma$ consists of two points (see the beginning of this section for the assumption on $\Sigma$).
Now suppose $\dim\Sigma\ge 1$ and $k\ge 2$.
By Remark \ref{rem:ineq} and Lemma \ref{lem:ind}(2), $\{\xi_i'\}_{i=2}^k$ in $\Sigma_{\xi_1}$ is a $(c(\varepsilon),\varkappa(\delta))$-noncritical collection with a regular direction $\eta'$.
Thus the induction completes the proof, provided $\delta\ll\varepsilon$.

Next we prove (2) by induction on $\dim\Sigma$.
We may assume $k\ge2$ by Remark \ref{rem:ineq}.
For $1\le j\le k$, set
\begin{equation*}
X_j:=\left\{x\in\Sigma\;\middle|\;
\begin{gathered}
|x\xi_i|=\pi/2\ (2\le i\le j),\\
|x\xi_i|\ge\pi/2\ (j<i\le k),\\
|x\eta|\ge\pi/2+c(\varepsilon)
\end{gathered}
\right\}.
\end{equation*}
We will show $X_j\neq\emptyset$ inductively.

Let us first show $X_1\neq\emptyset$.
By Remark \ref{rem:ineq} we have $|\xi_1\eta|>\pi/2+\varepsilon$ and $|\xi_1\xi_i|>\pi/2-\delta$ for any $i\ge 2$.
In other words, $\xi_1$ almost satisfies the inequalities in the definition of $X_1$.
We slightly move it to obtain a point of $X_1$.
As we have seen above, $\{\xi_i'\}_{i=2}^k$ is $(c(\varepsilon),\varkappa(\delta))$-noncritical in $\Sigma_{\xi_1}$ and $\eta'$ is a regular direction.
In particular $|\xi_i'\eta'|>\pi/2+c(\varepsilon)$ by Remark \ref{rem:ineq}, which means that moving $\xi_1$ toward $\eta$ along the shortest path increases the distance to $\xi_i$ with velocity at least $c(\varepsilon)$.
This estimate holds until the distance to $\xi_i$ exceeds $\pi/2$ within time $\varkappa(\delta)$.
Thus the resulting point lies in $X_1$.

Next we show $X_j\neq\emptyset$ by assuming $X_{j-1}\neq\emptyset$.
Let $x_j$ be a closest point to $\xi_j$ in $X_{j-1}$.
We show $x_j\in X_j$, i.e.\ $|x_j\xi_j|=\pi/2$.
As before, $\{\xi_i'\}_{i=2}^k$ is $(c(\varepsilon),\varkappa(\delta))$-noncritical in $\Sigma_{x_j}$ and $\eta'$ is a regular direction.
This holds in a small neighborhood of $x_j$.
Thus by Lemma \ref{lem:open}, Remark \ref{rem:ineq}, and the inductive hypothesis, the map $f=(|\xi_2\cdot|,\dots,|\xi_k\cdot|)$ is $c(\varepsilon)$-open near $x_j$, provided $\delta\ll\varepsilon$.
Furthermore by the inductive hypothesis, there exists $v\in\Sigma_{x_j}$ such that
\[|v\xi_j'|<\pi/2-c(\varepsilon),\quad|v\xi_i'|=\pi/2,\quad|v\eta'|>\pi/2+c(\varepsilon)\]
for any $i\neq j$ ($\ge2$).
Hence by Lemma \ref{lem:dir}, there exists $x\in\Sigma$ near $x_j$ such that
\[|x\xi_j|<|x_j\xi_j|,\quad|x\xi_i|=|x_j\xi_i|,\quad|x\eta|>|x_j\eta|\]
for any $i\neq j$ ($\ge2$).
This contradicts the choice of $x_j$ if $|x_j\xi_j|>\pi/2$.

Therefore $X_k\neq\emptyset$.
Let $v$ be a closest point to $\xi_1$ in $X_k$.
It suffices to show $|v\xi_1|<\pi/2-c(\varepsilon)$.
Suppose the contrary: $|v\xi_1|>\pi/2-\delta$, where $\delta=c(\varepsilon)$ will be determined by the following argument.
By Lemma \ref{lem:ind}(2), $\{\xi_i'\}_{i=1}^k$ in $\Sigma_v$ is a $(c(\varepsilon),\varkappa(\delta))$-noncritical collection with a regular direction $\eta'$.
If $k=\dim\Sigma+1$, this contradicts (1), provided $\delta\ll\varepsilon$.
If $k<\dim\Sigma+1$, the same argument as above shows that there exists $w\in\Sigma$ such that
\[|w\xi_1|<|v\xi_1|,\quad|w\xi_i|=|v\xi_i|,\quad|w\eta|>|v\eta|\]
for any $i\ge2$, provided $\delta\ll\varepsilon$.
This contradicts the choice of $v$.
\end{proof}

\section{Local properties}\label{sec:loc}

In this section we study the local properties of noncritical maps and prove Theorems \ref{thm:open} and \ref{thm:hure}.
As before, we use the notation $\varepsilon$, $\delta$, $c(\varepsilon)$, and $\varkappa(\delta)$.

Throughout this section, $U$ denotes a tiny ball in a GCBA space $X$.
Note that for any $a,x,y\in U$, we have $\tilde\angle axy+\tilde\angle ayx\le\pi$, where $\tilde\angle$ denotes the comparison angle.
In particular if $|ax|\ge|ay|$, we obtain $\angle axy\le\tilde\angle axy\le\pi/2$.
We will often use this observation.

Let $f=(|a_1\cdot|,\dots,|a_k\cdot|)$ be an $(\varepsilon,\delta)$-noncritical map at $p\in U$ in the sense of Definition \ref{dfn:nc} and let $b\in U$ be as in the condition (2).
In the terminology of the previous section, $\{a_i'\}_{i=1}^k$ is an $(\varepsilon,\delta)$-noncritical collection in $\Sigma_p$ with a regular direction $b'$.
We also call $b'$ a \textit{regular direction} of $f$ at $p$ (note that this notation specifies not only the direction $b'$ but also the point $b$).
By the upper semicontinuity of angle and the local geodesic completeness, there exists a neighborhood $V\subset U$ of $p$ such that $f$ is an $(\varepsilon,\delta)$-noncritical map on $V$ with a regular direction $b'$.
We call $V$ a \textit{regular neighborhood} of $f$ at $p$ with respect to $b$.

If we want to estimate the size of a regular neighborhood, we use the following alternative definition.
Let $\rho$ be a positive number less than the radius of $U$.

\begin{dfn}\label{dfn:nc'}
Let $a_i\in U$ ($1\le i\le k$).
We say that $f=(|a_1\cdot|,\dots,|a_k\cdot|)$ is an \textit{$(\varepsilon,\delta,\rho)$-noncritical map} at $p\in U$ if the following hold:
\begin{enumerate}
\item $|a_ip|>\rho$ and
\[\tilde\angle a_ipx+\tilde\angle a_jpx<3\pi/2+\delta\]
for any $x\in B(p,\rho)\setminus\{p\}$ and $1\le i\neq j\le k$;
\item there exists $b\in U$ such that $|bp|>\rho$ and
\[\tilde\angle a_ipx+\tilde\angle bpx<3\pi/2-\varepsilon\]
for any $x\in B(p,\rho)\setminus\{p\}$ and $1\le i\le k$.
\end{enumerate}
\end{dfn}

In this case as well, we call $b'$ a \textit{regular direction} of  $f$ at $p$.

\begin{rem}\label{rem:nc'}
By the local geodesic completeness and the monotonicity of comparison angle, it suffices to consider the case $x\in B(p,\rho)\setminus B(p,\sigma)$ in the above inequalities, where $0<\sigma<\rho$.
\end{rem}

The above definition is equivalent to Definition \ref{dfn:nc} in the following sense:

\begin{lem}\label{lem:nc'1}
A map $f=(|a_1\cdot|,\dots,|a_k\cdot|)$ is $(\varepsilon,\delta)$-noncritical at $p\in U$ if and only if it is $(\varepsilon,\delta,\rho)$-noncritical at $p$ for some $\rho>0$.
\end{lem}

\begin{proof}
The ``if'' part follows from the local geodesic completeness and the monotonicity of angle.
We show the ``only if'' part.
By definition, we have $|a_i'\xi|+|a_j'\xi|<3\pi/2+\delta$ for any $\xi\in\Sigma_p$.
Then the first variation formula implies that $\tilde\angle a_ipx+\tilde\angle a_jpx<3\pi/2+\delta$ for any $x$ sufficiently near $p$.
Hence the first condition of Definition \ref{dfn:nc'} holds for some $\rho>0$.
Similarly, the second condition holds.
\end{proof}

The size of a regular neighborhood can be estimated as follows:

\begin{lem}\label{lem:nc'2}
Let $f=(|a_1\cdot|,\dots,|a_k\cdot|)$ be an $(\varepsilon,\delta,\rho)$-noncritical map at $p$.
Then it is $(\varepsilon/2,\varkappa(\delta),\rho/2)$-noncritical on $B(p,\rho\delta)$.
\end{lem}

\begin{proof}
By definition, we have $\tilde\angle a_ipx+\tilde\angle a_jpx<3\pi/2+\delta$ for any $x\in B(p,\rho)\setminus\{p\}$, where $|a_ip|,|a_jp|>\rho$.
Hence if $q\in B(p,\rho\delta)$ and  $x\in B(q,\rho/2)\setminus B(q,\rho/3)$, we have $\tilde\angle a_iqx+\tilde\angle a_jqx<3\pi/2+\varkappa(\delta)$.
By Remark \ref{rem:nc'}, this implies the first condition of Definition \ref{dfn:nc'} for $q$.
Similarly, one can show the second condition.
\end{proof}

Now we will investigate the local properties of noncritical maps.
First we prove Theorem \ref{thm:open}.

\begin{proof}[Proof of Theorem \ref{thm:open}]
Let $f=(|a_1\cdot|,\dots,|a_k\cdot|)$ be an $(\varepsilon,\delta)$-noncritical map at $p$ with a regular direction $b'$ and let $V$ be a regular neighborhood.
Then $k\le\dim T_p$ by Proposition \ref{prop:dir}(1).
Furthermore $f$ is $c(\varepsilon)$-open on $V$ by Lemma \ref{lem:open}, Remark \ref{rem:ineq}, and Proposition \ref{prop:dir}(2).

Now assume $k=\dim T_p$.
We show $f$ is injective near $p$.
Suppose the contrary and let $x,y\in V$ be distinct points sufficiently close to $p$ such that $f(x)=f(y)$.
In particular we have $\angle a_ixy\le \pi/2$ for any $i$ as observed at the beginning of this section.
We may assume $|bx|\ge|by|$, which implies $\angle bxy\le\pi/2$ as well.
Hence $y'\in\Sigma_x$ satisfies the assumption (1) of Lemma \ref{lem:ind} for a noncritical collection $\{a_i'\}_{i=1}^k$ and its regular direction $b'$ (in particular we may assume $k\ge2$).
Thus $\{(a_i')'\}_{i=1}^k$ in $\Sigma_{y'}$ is a noncritical collection with a regular direction $(b')'$.
Therefore by Proposition \ref{prop:dir}(1) we have $k\le\dim\Sigma_{y'}+1\le\dim\Sigma_x$.
On the other hand, if $x$ is sufficiently close to $p$, we have $\dim T_x\le k$ by the upper semicontinuity of local dimension.
This is a contradiction.
\end{proof}

From now on, we assume $k<\dim T_p$.
We first observe that the fiber of $f$ through $p$ is not a singleton.

\begin{lem}\label{lem:fibr}
Let $f=(|a_1\cdot|,\dots,|a_k\cdot|)$ be an $(\varepsilon,\delta)$-noncritical map at $p\in U$, where $k<\dim T_p$.
Then there exists a point of $f^{-1}(f(p))$ arbitrarily close to $p$.
\end{lem}

\begin{proof}
Let $b'$ be a regular direction of $f$ at $p$ and $V$ a regular neighborhood.
We may assume that there exist $x_j\neq y_j\in V$ converging to $p$ such that $f(x_j)=f(y_j)$; otherwise $f$ is a bi-Lipschitz open embedding near $p$, which contradicts $k<\dim T_p$.
In particular we have $\angle a_ix_jy_j\le\pi/2$ for any $i$ as before.
We may further assume $|bx_j|\ge|by_j|$ and hence $\angle bx_jy_j\le\pi/2$.

Extend the shortest path $y_jx_j$ beyond $x_j$ to a shortest path $y_jz_j$ of fixed length.
The above estimates imply $\angle a_ix_jz_j\ge\pi/2$ and $\angle bx_jz_j\ge\pi/2$.
We may assume $z_j$ converges to $z\neq p$.
By the upper semicontinuity of angle, we have $\angle a_ipz\ge\pi/2$ and $\angle bpz\ge\pi/2$.

For $0\le j\le k$, set 
\begin{equation*}
X_j:=\left\{x\in\Sigma_p\;\middle|\;
\begin{gathered}
|xa_i'|=\pi/2\ (1\le i\le j),\\
|xa_i'|\ge\pi/2\ (j<i\le k),\\
|xb'|\ge\pi/2
\end{gathered}
\right\}.
\end{equation*}
We have now shown $z'\in X_0$.
Furthermore an inductive argument similar to the proof of Proposition \ref{prop:dir}(2) shows $X_j\neq\emptyset$ for any $j\ge 1$.
In particular for $\xi\in X_k$ we have $f'(\xi)=0$.
Thus the claim follows from Lemma \ref{lem:dir}.
\end{proof}

Next we construct a local neighborhood retraction to the fiber of a noncritical map.
Such a retraction for a strainer map (with much nicer properties) was constructed in \cite[9.1]{LN:geod}.
Our construction is inspired by a similar argument in CBB geometry \cite[6.15]{K:stab}.
Here we use Definition \ref{dfn:nc'} to specify the size of a regular neighborhood (Lemma \ref{lem:nc'2}).

\begin{prop}\label{prop:ret}
Let $f=(|a_1\cdot|,\dots,|a_k\cdot|)$ be $(\varepsilon,\delta,\rho)$-noncritical at $p\in U$.
Then, for any $0<r<\rho\delta$, there exists a continuous map
\[R:B(p,r)\to B(p,Lr)\cap f^{-1}(f(p))\]
that is the identity on $B(p,r)\cap f^{-1}(f(p))$, where $L=c(\varepsilon)^{-1}$.
\end{prop}

\begin{proof}
Let $b'$ be a regular direction of $f$ at $p$.
To simplify the notation, we set $f_i=|a_i\cdot|-|a_ip|$ for any $i$.
Let $s=c(\varepsilon)$, which will be determined later as well as $L$.
Define
\begin{align*}
&\Pi_+^s:=\{x\in U\mid f_i(x)\ge sf_j(x)\text{ for any }i\neq j\},\\
&\Pi_-:=\{x\in U\mid f_i(x)\le0\text{ for any }i\},\\
&\Pi:=\Pi_+^s\cap\Pi_-=f^{-1}(f(p))\cap U.
\end{align*}
Note that $\Pi_-$ is convex since $f_i$ is convex.
If $k=1$, we consider $\Pi_+:=\{x\in U\mid f_1(x)\ge0\}$ instead of $\Pi_+^s$.

First we construct a map $R_1:B(p,r)\to B(p,Lr)\cap\Pi_+^s$ as follows:
for any $x\in B(p,r)$, let $R_1(x)$ be the closest point to $x$ on the intersection of the shortest path $xb$ and $\Pi_+^s$.
Let us check this definition works.
By Lemmas \ref{lem:nc'1} and \ref{lem:nc'2}, $f$ is a $(c(\varepsilon),\varkappa(\delta))$-noncritical map with a regular direction $b'$ on $B(p,c(\varepsilon)^{-1}r)$.
In particular $|a_i'b'|>\pi/2+c(\varepsilon)$ on $B(p,c(\varepsilon)^{-1}r)$ for any $i$ by Remark \ref{rem:ineq}.
This means if one moves $x$ toward $b$ along the shortest path, the value of $f_i$ increases with velocity $>c(\varepsilon)$.
We show it reaches $\Pi_+^s$ within time $2r/c(\varepsilon)$ (the case $k=1$ is obvious).
Let $\gamma(t)$ be the shortest path $xb$ with arclength parameter $t$.
Then $-r+c(\varepsilon)t\le f_i(\gamma(t))\le r+t$.
In particular if $t=2r/c(\varepsilon)$, we have
\[\frac{f_j(\gamma(t))}{f_i(\gamma(t))}\le\frac{r+2rc(\varepsilon)^{-1}}r=1+2c(\varepsilon)^{-1}.\]
Hence $R_1$ can be defined for $s=(1+2c(\varepsilon)^{-1})^{-1}$ and $L=2c(\varepsilon)^{-1}+1$.

Furthermore any $y\in B(p,Lr)$ on the shortest path $xb$ beyond $R_1(x)$ lies in the interior of $\Pi_+^s$ (the case $k=1$ is obvious).
Indeed,
\begin{align*}
f_i(y)&>f_i(R_1(x))+c(\varepsilon)|R_1(x)y|\\
&\ge sf_j(R_1(x))+c(\varepsilon)|R_1(x)y|\\
&\ge sf_j(y)+(c(\varepsilon)-s)|R_1(x)y|>sf_j(y),
\end{align*}
where we take $s$ smaller if necessary.
Together with the uniqueness of shortest paths, this shows that $R_1$ is continuous.
Clearly $R_1$ is the identity on $\Pi_+^s$.

Next we construct a map $R_2:B(p,Lr)\cap \Pi_+^s\to\Pi_-$ as follows:
for any $x\in B(p,Lr)\cap \Pi_+^s$, let $R_2(x)$ be the closest point to $x$ on $\Pi_-$.
Since $\Pi_-$ is convex, triangle comparison shows that the closest point is uniquely determined.

The map $R_2$ does not increase the distance to $p$.
In particular $R_2(x)\in B(p,Lr)$.
Let $y=R_2(x)$.
Since $y$ is closest to $x$ on the convex set $\Pi_-$, the first variation formula implies $\angle pyx\ge\pi/2$.
Thus we have $\tilde\angle pyx\ge\pi/2$ and $\tilde\angle pxy\le\pi/2$, which mean $|py|\le|px|$.

We show that $R_2(x)\in\Pi$ (the case $k=1$ is obvious).
Suppose $y=R_2(x)\notin\Pi$.
We may assume $f_i(y)<0$ for some $i$.
We first observe that $\tilde\angle a_ixy<\pi/2-c(\varepsilon)$.
By the $c(\varepsilon)$-openness of $f$, there exists $z\in\Pi$ such that $c(\varepsilon)|xz|\le|f(x)|$, where $f(p)=0$.
Since $x\in\Pi_+^s$, we have $|f(x)|\le c(\varepsilon)^{-1}f_i(x)$.
Thus
\[f_i(x)\ge c(\varepsilon)|xz|\ge f_i(y)+c(\varepsilon)|xy|,\]
where the second inequality follows from the definition of $y$.
This implies $\tilde\angle a_ixy<\pi/2-c(\varepsilon)$.

Therefore $\angle a_ixy<\pi/2-c(\varepsilon)$.
Extend the shortest path $a_ix$ beyond $x$ to a shortest path $a_i\bar a_i$ with $|\bar a_ix|=\rho/2$.
Then we have $\tilde\angle\bar a_ixy\ge\angle\bar a_ixy>\pi/2+c(\varepsilon)$ and hence $\angle\bar a_iyx\le\tilde\angle\bar a_iyx<\pi/2-c(\varepsilon)$.
This means that moving $y$ toward $\bar a_i$ decreases the distance to $x$ with velocity at least $c(\varepsilon)$.

On the other hand, since $|xy|\ll\rho$ we have $\tilde\angle a_iy\bar a_i>\pi-\varkappa(\delta)$.
By Definition \ref{dfn:nc'}(1) we have $\angle a_jy\bar a_i\le\tilde\angle a_jy\bar a_i<\pi/2+\varkappa(\delta)$ for any $j\neq i$.
This means that moving $y$ toward $\bar a_i$ increases the value of $f_j$ with velocity at most $\varkappa(\delta)$.

Let $y_1$ be a point on the shortest path $y\bar a_i$ sufficiently close to $y$.
Then the assumption on $y$ and the above observations yield
\[f_i(y_1)<0,\quad f_j(y_1)<\varkappa(\delta)|yy_1|,\quad |xy_1|<|xy|-c(\varepsilon)|yy_1|\]
for any $j\neq i$.
By the $c(\varepsilon)$-openness of $f$, we can find $y_2\in\Pi_-$ such that $c(\varepsilon)|y_2y_1|\le\varkappa(\delta)|yy_1|$.
Since $\varkappa(\delta)\ll c(\varepsilon)$, we have $|xy_2|\le|xy_1|+|y_1y_2|<|xy|$.
This is a contradiction to the choice of $y$.
Therefore $y=R_2(x)\in\Pi$.

The desired retraction $R$ is now obtained as the composition of $R_1$ and $R_2$.
\end{proof}

\begin{rem}
One can also construct $R_2$ by using the gradient flows of semiconcave functions as in the CBB case \cite[6.15]{K:stab}.
The existence and uniqueness of such flows on GCBA spaces were shown in \cite{L:open} (cf.\ \cite{Pet:sc}).

We briefly recall the definition of the gradient (see \cite{L:open} or \cite{Pet:sc} for more details).
Let $f$ be a semiconcave function defined on an open subset $U$ of a GCBA space.
Note that the directional derivative $f'$ is well-defined in any direction and extended to the tangent cone by positive homogeneity.
The \textit{gradient} $\nabla_p f\in T_p$ of $f$ at $p\in U$ is characterized by the following two properties:
\begin{enumerate}
\item $f'(v)\le\langle\nabla_pf,v\rangle$ for any $v\in T_p$;
\item $f'(\nabla_pf)=|\nabla_pf|^2$.
\end{enumerate}
Here the absolute value denotes the distance from the vertex $o$ of $T_p$ and the scalar product is defined by $\langle u,v\rangle:=|u||v|\cos\angle(u,v)$ for $u,v\in T_p$.
More specifically, if $\max f'|_{\Sigma_p}>0$ then $\nabla_pf=f'(\xi_{\max})\xi_{\max}$, where $\xi_{\max}\in\Sigma_p$ is the unique maximum point of $f'|_{\Sigma_p}$; otherwise $\nabla_pf=o$.
The gradient flow of $f$ is the flow along the gradient of $f$.

The alternative construction of $R_2$ is as follows.
Set a concave function $F:=\min\{-f_i,0\}$ and consider the gradient flow $\Phi_t$ of $F$ on $B(p,c(\varepsilon)^{-1}r)$, where $t\ge0$.
We will define $R_2:=\Phi_T$ for some sufficiently large $T>0$.

Clearly $\Phi_t$ fixes $\Pi_-$.
Suppose $x\notin\Pi_-$.
Let $I(x)$ be the set of indices $1\le i\le k$ such that $F(x)=-f_i(x)$.
Then it follows from the noncriticality of $f$ that
\begin{enumerate}
\item if $i\in I(x)$ then $f_i'(\nabla_xF)<-c(\varepsilon)$;
\item if $i\notin I(x)$ then $f_i'(\nabla_xF)>-\varkappa(\delta)$.
\end{enumerate}
Indeed, let $\bar b'$ (resp.\ $\bar a_i'$) be an antipode of $b'$ (resp.\ $a_i'$) in $\Sigma_x$.
By the noncriticality of $f$ we have $f_i'(\bar b')<-c(\varepsilon)$ for any $i$.
Fix $1\le i\le k$.
If $i\in I(x)$, this implies $f_i'(\nabla_xF)<-c(\varepsilon)$.
If $i\notin I(x)$, then we have
\[f_i'(\nabla_xF)=-\langle a_i',\nabla_xF\rangle\ge\langle\bar a_i',\nabla_xF\rangle\ge F_x'(\bar a_i')>-\varkappa(\delta),\]
where the first inequality follows from $|a_i'\bar a_i'|=\pi$, the second one follows from the property of the gradient, and the last one follows from the noncriticality of $f$.

We show that $\Phi_t$ pushes $B(p,Lr)\cap \Pi_+^s$ to $\Pi$ inside $\Pi_+^s$ within time $T=Lr/c(\varepsilon)$.
Let $x\in B(p,Lr)\cap \Pi_+^s$.
It suffices to show that $y=\Phi_t(x)$ lies in $\Pi_+^s$ for sufficiently small $t>0$.
We may assume $x$ is on the boundary of $\Pi_+^s$, that is, $f_i(x)=sf_j(x)$ for some $i\neq j$.
This implies $i\notin I(x)$ and $j\in I(x)$.
By the observation above, we have
\begin{align*}
f_i(y)&>f_i(x)-\varkappa(\delta)t\\
&=sf_j(x)-\varkappa(\delta)t\\
&>s(f_j(y)+c(\varepsilon)t)-\varkappa(\delta)t>sf_j(y),
\end{align*}
which means $y\in\Pi_+^s$.

Hence we can define $R_2:=\Phi_T$ on $B(p,Lr)\cap \Pi_+^s$ (the image of $R_2$ is contained in $B(p,Lr+Lr/c(\varepsilon))$, so we need to replace the constant $L$ with $L+L/c(\varepsilon)$ in the statement of Proposition \ref{prop:ret}).
\end{rem}

Since any metric ball contained in a tiny ball is contractible, we have:

\begin{cor}[cf.\ {\cite[1.11]{LN:geod}}]\label{cor:ret}
Let $f:U\to\mathbb R^k$ be $(\varepsilon,\delta,\rho)$-noncritical at $p\in U$.
Then for any $0<r<\rho\delta$, $B(p,r)\cap f^{-1}(f(p))$ is contractible inside $B(p,Lr)\cap f^{-1}(f(p))$, where $L=c(\varepsilon)^{-1}$.
In particular, $f$ has locally uniformly contractible fiber on $B(p,r)$ (see Definition \ref{dfn:lucf}).
\end{cor}

The following theorem is a direct consequence of Theorem \ref{thm:glo}, Theorem \ref{thm:open}, and Corollary \ref{cor:ret}.

\begin{thm}\label{thm:hure'}
Let $f:U\to\mathbb R^k$ be a distance map defined on a tiny ball $U$ that is noncritical on an open subset $V\subset U$.
Let $K\subset f(V)$ be a compact set such that $f^{-1}(K)\cap V$ is compact.
Then $f:f^{-1}(K)\cap V\to K$ is a Hurewicz fibration.
\end{thm}

To prove Theorem \ref{thm:hure}, we need to control the boundary of a ball in the fiber.

\begin{lem}\label{lem:comp}
Let $f=(|a_1\cdot|,\dots,|a_k\cdot|)$ be an $(\varepsilon,\delta)$-noncritical map at $p\in U$, where $k<\dim T_p$.
Then for any sufficiently small $r>0$, the map $(f,|p\cdot|)$ is $(c(\varepsilon),\varkappa(\delta))$-noncritical on $\partial B(p,r)\cap f^{-1}(f(p))$.
\end{lem}

\begin{rem}\label{rem:comp}
By Lemma \ref{lem:fibr} and Corollary \ref{cor:ret}, $\partial B(p,r)\cap f^{-1}(f(p))\neq\emptyset$ for any sufficiently small $r$.
\end{rem}

\begin{proof}
Let $x\in\partial B(p,r)\cap f^{-1}(f(p))$.
In particular $\angle a_ipx\le\pi/2$.
Let $\bar p'$ be an arbitrary antipode of $p'$ in $\Sigma_x$.
By the upper semicontinuity of angle and the local geodesic completeness, we have $|a_i'\bar p'|<\pi/2+\delta$ if $r$ is small enough.
In other words $\overline{|a_i'p'|}<\pi/2+\delta$.

It remains to show that there exists $\eta\in\Sigma_x$ such that
\[\overline{|a_i'\eta|}<\pi/2-c(\varepsilon),\quad\overline{|p'\eta|}<\pi/2-c(\varepsilon).\]
We may assume that $f$ is $(\varepsilon,\delta)$-noncritical on $B(p,r)$.
In particular there exists $\xi\in\Sigma_x$ such that $\overline{|a_i'\xi|}<\pi/2-\varepsilon$.
We move $\xi$ toward $\bar p'$ to get $\eta$ as above.

Fix an antipode $\bar p'$ of $p'$ in $\Sigma_x$.
Recall that $p$ is a $(1,\delta)$-strainer at $x$ if $r$ is small enough (see \cite[7.3]{LN:geod}).
In particular $\overline{|p'\bar p'|}<2\delta$ (see \cite[6.3]{LN:geod}).
Let $\bar a_i'$ be an arbitrary antipode of $a_i'$ in $\Sigma_x$.
Since $x\in f^{-1}(f(p))$, we have $|a_i'p'|\le\pi/2$ and thus $|\bar a_i'p'|\ge\pi/2$.
Together with $\overline{|p'\bar p'|}<2\delta$, this implies $|\bar a_i'\bar p'|<\pi/2+2\delta$ (conversely $|\bar a_i'\bar p'|\ge\pi/2-\delta$ since $|a_i'\bar p'|<\pi/2+\delta$).

Now we move $\xi$ toward $\bar p'$ in $\Sigma_x$.
By triangle comparison, we have
\[\cos\angle\bar p'\xi\bar a_i'\ge\frac{\cos|\bar p'\bar a_i'|-\cos|\xi\bar p'|\cos|\xi\bar a_i'|}{\sin|\xi\bar p'|\sin|\xi\bar a_i'|}.\]
Since $|\xi\bar a_i'|<\pi/2-\varepsilon$ and $|\bar p'\bar a_i'|<\pi/2+2\delta$, we have $\angle\bar p'\xi\bar a_i'<\pi/2$ as long as $|\xi\bar p'|>\pi/2+c(\varepsilon)$.
Hence we can find $\eta_1\in\Sigma_x$ such that $|\eta_1\bar a_i'|\le|\xi\bar a_i'|<\pi/2-\varepsilon$ and $|\eta_1\bar p'|\le\pi/2+c(\varepsilon)$ (if $|\xi\bar p'|\le\pi/2+c(\varepsilon)$, we can take $\eta_1:=\xi$).
Moving $\eta_1$ further toward $\bar p'$ by distance $2c(\varepsilon)$ if necessary, we obtain $\eta\in\Sigma_x$ such that $|\eta\bar p'|\le\pi/2-c(\varepsilon)$ and $|\eta\bar a_i'|<\pi/2-c(\varepsilon)$.
Since $\overline{|p'\bar p'|}<2\delta$, this completes the proof.
\end{proof}

We are now in a position to prove Theorem \ref{thm:hure}.
For $0<r_1<r_2$, we set
\[A_p[r_1,r_2]:=\bar B(p,r_2)\setminus B(p,r_1),\quad A_p[r_1,r_2):=B(p,r_2)\setminus B(p,r_1).\]

\begin{proof}[Proof of Theorem \ref{thm:hure}]

Let $L$ be the constant of Corollary \ref{cor:ret}.
By Lemma \ref{lem:comp}, $(f,|p\cdot|)$ is noncritical on $A_p[r/2,Lr]\cap f^{-1}(f(p))$ for sufficiently small $r>0$.
Let $B$ be a sufficiently small closed ball around $f(p)$.
By the openness of $(f,|p\cdot|)$ (Theorem \ref{thm:open}) and Remark \ref{rem:comp}, $B\times[r/2,Lr]$ is contained in the image of $(f,|p\cdot|)$.
Therefore by Theorem \ref{thm:hure'}, $(f,|p\cdot|)$ is a Hurewicz fibration over $B\times[r/2,Lr]$.

Let $\mathring B$ denote the interior of $B$.
We will apply Theorem \ref{thm:loc} to the map $f:f^{-1}(\mathring B)\cap B(p,Lr)\to\mathring B$.
It suffices to show that any fiber of $f$ is contractible.
Set $g=|p\cdot|$.
By the previous paragraph, $g:f^{-1}(b)\cap A_p[r/2,Lr]\to[r/2,Lr]$ is a Hurewicz fibration for any $b\in\mathring B$.
Let $h_t$ $(0\le t\le 1)$ be a homotopy crushing $[r/2,Lr)$ to $\{r/2\}$ linearly.
Let $H_t:f^{-1}(b)\cap A_p[r/2,Lr)\to f^{-1}(b)\cap A_p[r/2,Lr)$ be a lift of $h_t\circ g$, where the lift of $h_0\circ g$ is given by the identity.
We define
\begin{equation*}
\tilde H_t:=
\begin{cases}
\hfil H_t & \text{on}\quad A_p[r,Lr)\cap f^{-1}(b)\\
H_{(2|p\cdot|/r-1)t} & \text{on}\quad A_p[r/2,r)\cap f^{-1}(b)\\
\hfil\mathrm{id} & \text{on}\quad B(p,r/2)\cap f^{-1}(b).
\end{cases}
\end{equation*}
Then $\tilde H_t$ pushes $f^{-1}(b)\cap B(p,Lr)$ into $f^{-1}(b)\cap B(p,r)$.
Composing this with the retraction of Corollary \ref{cor:ret}, we obtain a homotopy crushing $f^{-1}(b)\cap B(p,Lr)$ to $p$ inside $f^{-1}(b)\cap B(p,Lr)$.
This completes the proof.
\end{proof}

Finally we prove Corollary \ref{cor:sph}.
The proof is an easy application of the second half of Theorem \ref{thm:open}.
Strictly speaking we need its global version for CAT($0$) spaces, which easily can be verified.

\begin{rem}\label{rem:bilip}
Under the assumption of Corollary \ref{cor:sph}, the diameter of $\Sigma$ (with respect to the original metric) is less than $2\pi$; in particular the original metric is uniformly bi-Lipschitz equivalent to the $\pi$-truncation.
This is observed as follows.
Let $\{\xi_i\}_{i=1}^{n+1}$ and $\eta$ be as in the assumption.
Then Lemma \ref{lem:ind}(2) and Proposition \ref{prop:dir}(1) imply that they are $\pi/2$-dense in $\Sigma$, whereas Remark \ref{rem:ineq} shows that they are $\pi$-close to each other.
\end{rem}

\begin{proof}[Proof of Corollary \ref{cor:sph}]
Consider the Euclidean cone $K$ over $\Sigma$ and let $\gamma_i$ be the ray starting at the vertex $o$ in the direction $\xi_i$.
Let $f_i:K\to\mathbb R$ be the Busemann function with respect to $\gamma_i$:
\[f_i(v)=\lim_{t\to\infty}|\gamma_i(t),v|-t=-|v|\cos\angle(\xi_i,v),\]
where the absolute value denotes the distance from $o$.
Set $f:=(f_1,\dots,f_{n+1}):K\to\mathbb R^{n+1}$.
Observe that $f$ is a (normalized) limit of $(\varepsilon,\delta)$-noncritical maps to $\mathbb R^{n+1}$ defined on arbitrarily large neighborhoods of $o$, which are $c(\varepsilon)$-open embeddings by the global version of Theorem \ref{thm:open} for CAT($0$) spaces.
Therefore $f$ is a bi-Lipschitz homeomorphism.
Identify $\Sigma$ with the unit sphere in $K$ centered at $o$ and $\mathbb S^n$ with the unit sphere in $\mathbb R^{n+1}$ centered at $0$.
We may use the extrinsic metrics of $\Sigma$ and $\mathbb S^n$, that is, the restrictions of the metrics of $K$ and $\mathbb R^{n+1}$, respectively.
Define $\tilde f:\Sigma\to\mathbb S^n$ by $\tilde f(x):=f(x)/|f(x)|$.

Let us prove that $\tilde f$ is a bi-Lipschitz homeomorphism.
Note that $f$ maps a ray emanating from $o$ to a ray emanating from $0$.
Since $f$ is surjective, this implies that $\tilde f$ is also surjective.
The Lipschitz continuity of $\tilde f$ follows from that of $f$ by the following calculation:
\begin{align*}
|\tilde f(x)\tilde f(y)|&\le\left|\tilde f(x)\frac{f(y)}{|f(x)|}\right|+\left|\frac{f(y)}{|f(x)|}\tilde f(y)\right|\\
&=\frac{|f(x)f(y)|}{|f(x)|}+\left|\frac{|f(y)|}{|f(x)|}-1\right|\le2\frac{|f(x)f(y)|}{|f(x)|},
\end{align*}
where $x,y\in\Sigma$.
Since $f$ maps a ray to a ray, we have $\tilde f^{-1}(x)=f^{-1}(x)/|f^{-1}(x)|$ for any $x\in\mathbb S^n$.
Hence the Lipschitz continuity of $\tilde f^{-1}$ follows from that of $f^{-1}$ by the same calculation as above.
\end{proof}

\begin{ack}
I would like to thank Professor Koichi Nagano for answering my question regarding Example \ref{ex:iso}.
I am also grateful to the referee for valuable comments.
\end{ack}

\end{document}